\documentclass{amsart}
\usepackage{amsmath,mathrsfs}
\usepackage{amssymb,cite}
\usepackage{amsfonts}

\newtheorem{theorem}{Theorem}
\theoremstyle{plain}

\newtheorem{corollary}{Corollary}

\newtheorem{definition}{Definition}

\newtheorem{lemma}{Lemma}

\newtheorem{remark}{Remark}

\def\p{\mathfrak{p}}
\def\q{\mathfrak{q}}
\def\m{\mathfrak{m}}
\def\supp{\mathrm{supp}}
\def\co{\mathscr{O}}

\begin{document}
\title[]{Counting $w$-coprime $S$-integers and $S$-integral ideals\\ in positive characteristic}
\author{Si-Han Liu, Zhe-Cheng Liu, and Jia-Yan Yao}
\date{\today}
%\subjclass{Primary 11J81, 11J93; Secondary 11T55}

\begin{abstract}
Let $\mathbb{F}_q$ be the finite field with $q$ elements,
and let $K$ be an algebraic function field over $\mathbb{F}_q$
whose field of constants is $\mathbb{F}_q$.
Let $S$ be a finite\linebreak nonempty set of prime divisors over $K$, and let $\co_S$ be the ring of integers of~$K$ attached to $S$.
Let $w\geqslant 1$ be an integer. In this work we shall count
 $w$-coprime $S$-integers and $S$-integral ideals, and our proofs are a combination of analytic methods,
the Riemann-Roch theorem, and the Weil theorem for function fields in positive characteristic.
\end{abstract}

\maketitle

\markboth{Counting $w$-coprime $S$-integers and $S$-integral ideals in positive characteristic}
{Si-Han Liu, Zhe-Cheng Liu, and Jia-Yan Yao}%

\subsection{Introduction\label{section1}}

It is well known that the probability for two natural numbers chosen at random to be coprime is equal to $\frac{6}{\pi^2}$ (see for example \cite[pp.\,114--115]{hualuogeng2}). This result has been generalized in two directions.
On the one hand, one can consider the points in $\mathbb{P}^1(\mathbb{Q})$, where each point can be identified with two coprime integers in $\mathbb{Z}$, and it has been shown (see for example \cite[p.\,71]{lang}) that
\begin{equation*}
N(B)=\frac{6B^2}{\pi^2}+O({B\log B}),
\end{equation*}
where for all integers $B>0$, we denote by $N(B)$ the numbers of all points $x\in \mathbb{P}^1(\mathbb{Q})$ with height $H(x)\leqslant B$.
Recall here that
for all $x=(x_1:x_2)\in \mathbb{P}^1(\mathbb{Q})$ with $x_1,x_2\in \mathbb{Z}$ and $\gcd(x_1,x_2)=1$, the height of $x$ is defined by
$$
H(x)=\max(|x_1|,|x_2|).
$$
More generally, for a finite extension $\mathbb{K}$ of $\mathbb{Q}$ and for all integers $n\geqslant 2$,
Schanuel established in \cite{schanuel}
the asymptotic formula for $N_\mathbb{K}(B)$ with $B>0$,
where $N_\mathbb{K}(B)$ is the number of points in $\mathbb{P}^{n-1}(\mathbb{K})$ with height $\leqslant B$.
An analogous result for algebraic function fields was observed firstly by Serre \cite{serre} and proved independently by Dipippo \cite{Dipi}  and Wan \cite{wandaqing}.

On the other hand, for all integers $n,m\geqslant 1$,  one can also consider
\begin{eqnarray*}
\hskip 1cm V_m(n)=\mathrm{Card}\big(\big\{(x_1,,...,x_m)\in \mathbb{N}^m: 1\leqslant x_j\leqslant n,\, \textrm{$x_1,...,x_m$~are coprime}\big\}\big).
\end{eqnarray*}
It was shown by Lehmer \cite{leh} in 1900 that $\lim\limits_{n\to\infty}\frac{V_m(n)}{n^m}=\frac{1}{\zeta(m)}$. See also \cite{nym}.

More generally, one can consider $w$-coprime integers defined as follows.

\begin{definition}
Let $w\geqslant 1$ be an integer. If $x_1,...,x_m\in\mathbb{N}\setminus\{0\}$ do not have any  common factor of the form $q^w$ with $q\geqslant 2$ a prime number, then we say that
$x_1,...,x_m$ are $w$-coprime.
\end{definition}

Let $V_m^w(n)$ be the number of points $(x_1,,...,x_m)\in \mathbb{N}^m$ such that $x_1,...,x_m$ are $w$-coprime and $1\leqslant x_j\leqslant n\, (1\leqslant j\leqslant m)$.
For $wm\geqslant 2$, Benkoski \cite{ben}  showed in 1976 that $\lim\limits_{n\to\infty}\frac{V_m^w(n)}{n^m}=\frac{1}{\zeta(wm)}$.
It is worth pointing out that $V_1^2(n)$ is the number of squarefree integers among $1,2,...,n$, and that $x_1,...,x_m\in\mathbb{N}\setminus\{0\}$ are coprime if and only if they are $1$-coprime. Some further studies on this subject can be found, for instance, in \cite{hu,ah1,toth}.

As a natural generalization, one can consider $w$-coprime algebraic integers.
Let~$\mathbb{K}$ be a finite extension of $\mathbb{Q}$ with $d=[\mathbb{K}:\mathbb{Q}]$.
Let $\mathscr{O}_{\mathbb{K}}$ be the ring of algebraic integers of $\mathbb{K}$,
and ${\bf E}=\{{\bf e_i}\}_{i=1}^d$ a $\mathbb{Z}$-basis for $\mathscr{O}_{\mathbb{K}}$. For all integers $B>0$, define
\begin{equation*}
    \mathscr{O}_{\mathbb{K}}(B,{\bf E})=\left\{\sum_{i=1}^d{n_i{\bf e_i}}\colon n_j\in [-B,B)\cap \mathbb{Z}\right\}.
\end{equation*}

\begin{definition}
Let $a_1,\dots,a_m$ be elements in $\mathscr{O}_{\mathbb{K}}$, and let $\langle a_1,\dots,a_m\rangle$ be the ideal generated by
$a_1,\dots,a_m$ in $\mathscr{O}_{\mathbb{K}}$.
\begin{itemize}
\item[{\rm (1)}] If $\langle a_1,...,a_m\rangle=\mathscr{O}_{\mathbb{K}}$,
then we say that $a_1,...,a_m$ are coprime.

\item[{\rm (2)}] Let $w\geqslant 1$ be an integer.
If $\langle a_1,...,a_m\rangle$ is not contained in $\q^w$ for any prime ideal $\q$ in~$\mathscr{O}_{\mathbb{K}}$,
then we say that $a_1,...,a_m$ are $w$-coprime.
\end{itemize}
\end{definition}

The set of all $w$-coprime $m$-tuples $(a_1,\dots,a_m)$  is denoted by $T_w$.
Note that $a_1,\dots,a_m$ are $1$-coprime if and only if they are coprime.
Ferraguti and Micheli \cite{fm} showed in 2016 that
\begin{equation*}
    \lim_{B\to+\infty}\frac{|T_1 \cap\mathscr{O}_{\mathbb{K}}(B,{\bf E})|}{(2B)^{dm}}=\frac{1}{\zeta_{\mathbb{K}}(m)},
\end{equation*}
where $\zeta_{\mathbb{K}}(s)$ is the Dedekind zeta function of $\mathbb{K}$.
Note here that the above limit is independent of the
integral basis ${\bf E}$, and is called  the density of~$T_1$.
Micheli and Schnyder \cite{MS} considered the  density  in  function fields,
and one can compare their results with ours (see Remark \ref{remark5} below).

 Likewise one can consider $w$-coprime integral ideals instead of elements.

\begin{definition} Let $\mathfrak{a}_1,\dots,\mathfrak{a}_m$ be integral  ideals in $\mathscr{O}_{\mathbb{K}}$, and let $\langle \mathfrak{a}_1,\dots,\mathfrak{a}_m\rangle$ be the ideal generated by $\mathfrak{a}_1,\dots,\mathfrak{a}_m$ in $\mathscr{O}_{\mathbb{K}}$.
\begin{itemize}
\item[{\rm (1)}] If $\langle\mathfrak{a}_1,\dots,\mathfrak{a}_m\rangle=\mathscr{O}_{\mathbb{K}}$,
then we say that $\mathfrak{a}_1,\dots,\mathfrak{a}_m$ are coprime.

\item[{\rm (2)}] Let $w\geqslant 1$ be an integer.
If $\langle \mathfrak{a}_1,\dots,\mathfrak{a}_m\rangle$ is not contained in $\q^w$ for any prime ideal $\q$ in   $\mathscr{O}_{\mathbb{K}}$,
then we say that $\mathfrak{a}_1,\dots,\mathfrak{a}_m$ are $w$-coprime.
\end{itemize}
\end{definition}

 For all $x>0$,
 let $V(x)$ denote the number of integral ideals in $\mathscr{O}_{\mathbb{K}}$ with norm $\leqslant x$, and let $Q_m^w(x)$ denote the number of $m$-tuples $(\mathfrak{a}_1,\dots,\mathfrak{a}_m)$ of $w$-coprime integral ideals in
 $\mathscr{O}_{\mathbb{K}}$ with norm $\mathrm{N}(\mathfrak{a}_j)\leqslant x\ (1\leqslant j\leqslant m)$.
 Sittinger \cite{sittinger} showed in 2010 that
 $$
 \lim\limits_{x\to +\infty}\frac{Q_m^w(x)}{(V(x))^m}=\frac{1}{\zeta_{\mathbb{K}}(mw)}.
 $$
 Takeda \cite{takeda} proved that when assuming the Extended Lindel\"{o}f Hypothesis (ELH),
 the following asymptotic formula holds,
 \begin{equation}\label{eq 1}
     Q_m^1 (x)=\frac{c_{\mathbb{K}}x^m}{\zeta_{\mathbb{K}}(m)}+O(x^{m-1/2+\varepsilon}),
 \end{equation}
 where $c_{\mathbb{K}}>0$ is  a constant depending only on  $\mathbb{K}$. Later, Takeda \cite{takeda2} obtained the asymtotic formula without assuming the ELH, but the error term is worse than before.
 For other results of the same type over number fields, see for example \cite{sittinger2,tk,sd}.
 Similar problems have also been studied over the polynomial ring $\mathbb{F}_q[X]$, see for example \cite{md,ghl,ah2,lieb,sc}.

It is worth pointing out that all these results are related to zeta values.
Fujita and Yoshida \cite{fy} gave some applications of the zeta distribution (defined by the zeta values)
to number theory, especially several probabilistic proofs of the equality between the Riemann zeta function and some number-theoretic functions.

Inspired by all the above results, in the present work we shall proceed under the setting of general algebraic function fields in positive characteristic,
and exhibit results with relatively better error terms estimates.

We begin with some definitions and notation.

Fix a prime number $p$, and let $q=p^{\kappa}$, where $\kappa\geqslant 1$ is an integer.
Let $K$ be a fixed algebraic function field over $\mathbb{F}_q$, i.e.,
$K$ is a finite extension of some rational function field
$\mathbb{F}_q(X)$.
We also suppose that $\mathbb{F}_q$ is algebraically closed in $K$, and say that  $\mathbb{F}_q$ is the field of constants of $K$.
Then we have $K\cap\overline{\mathbb{F}}_q=\mathbb{F}_q$,
where $\overline{\mathbb{F}}_q$ is the algebraic closure of $\mathbb{F}_q$.

Let $M_K$ be the set of all prime divisors over $K$. We denote by $\mathrm{Div}_K$ the group of divisors over $K$,
which is just the free abelian group generated by $M_K$, i.e.,
\begin{equation*}
\mathrm{Div}_K=\bigg\{\sum_{\p\in M_K}c_{\p}\p: c_{\p}\in\mathbb{Z}, \text{and $c_{\p}=0$ except for a finite number of  $\p$'s}~\bigg\}.
\end{equation*}

For all $D=\sum\limits_{\p\in M_K}c_{\p}\p\in \mathrm{Div}_K$,
the degree and the norm of $D$ are defined by
\begin{eqnarray}
\deg D=\sum\limits_{\p\in M_K}c_{\p}\deg\p,\ \mathrm{and}\ \mathrm{N}(D)=q^{\deg D}.\label{2}
\end{eqnarray}
If for all $\p\in M_K$, we have $c_{\p}\geqslant 0$, then we say that $D$ is effective, and write $D\geqslant 0$. More generally, if $D_1,D_2\in \mathrm{Div}_K$ are such that $D_1-D_2\geqslant 0$,
then we write $D_1\geqslant D_2$.

Fix an integer $r\geqslant 0$, and let $S=\{\p_1,\p_2,...,\p_{r+1}\}$
be a subset of $M_K$. Define the ring of $S$-integers of $K$ by
\begin{equation*}
\co_S=\big\{a\in K: v_{\p}(a)\geqslant 0,\ \textrm{for all $\p\notin S$} \big\},
\end{equation*}
where $v_{\p}$ is the normalized valuation associated with $\p$.

\begin{definition}
Let $a_1,...,a_m$ be elements in $\co_S$, and let $\langle a_1,...,a_m\rangle_S$ be the ideal generated by $a_1,...,a_m$ in~$\co_S$.
\begin{itemize}
\item[{\rm (1)}] If $\langle a_1,...,a_m\rangle_S=\co_S$,
then we say that $a_1,...,a_m$ are coprime.

\item[{\rm (2)}] Let $w\geqslant 1$ be an integer.
If $\langle a_1,...,a_m\rangle_S$ is not contained in $\q^w$ for any prime ideal $\q$ in~$\co_S$,
then we say that $a_1,...,a_m$ are $w$-coprime.
\end{itemize}
As in the canonical case, $a_1,...,a_m$ are coprime if and only if they are $1$-coprime.
\end{definition}

Let $D\in \mathrm{Div}_K$. We say that $D$ is a {\em principal divisor} if $D=\sum\limits_{\p\in M_K}v_{\p}(a)\p\in \mathrm{Div}_K$,
for some $a\in K^{\times}$, and we denote $(a):=\sum\limits_{\p\in M_K}v_{\p}(a)\p\in \mathrm{Div}_K$.
By convention, we put $(0)=+\infty$, and write $(0)\geqslant D$, for all $D\in \mathrm{Div}_K$.

Let $\vec{N}=(N_1,...,N_{r+1})\in\mathbb{Z}^{r+1}$ be such that $N:=\sum\limits_{j=1}^{r+1}N_j\deg\p_j>0$. Put
\begin{eqnarray*}
D(\vec N)&=&\sum\limits_{j=1}^{r+1}N_j\p_j,\nonumber\\
L\big(D(\vec N)\big)&=&\big\{a\in K: (a)+D(\vec N)\geqslant 0\big\}\nonumber\\
&=&\big\{a\in\co_S: v_{\p_j}(a)\geqslant -N_j~(1\leqslant j\leqslant r+1)\big\},\\
V_m^w(\vec{N})&=&\mathrm{Card}\big(\big\{(a_1,...,a_m)\in \big(L(D(\vec N))\big)^m: \text{$a_1,...,a_m$ are $w$-coprime}\big\}\big).
\end{eqnarray*}
Note that the last one can be also reformulated as
$$
V_m^w(\vec{N})=\mathrm{Card}\left(\left\{(a_1,...,a_m)\in \co_S^m: \begin{array}{c}
              \textrm{$a_1,...,a_m$ are $w$-coprime, $v_{\mathfrak{p}_j}(a_i)\geqslant -N_j$} \\
              (1\leqslant i\leqslant m,\, 1\leqslant j\leqslant r+1)
              \end{array}\right\}\right).
$$

With the above definitions and notation, we have the following result.

\begin{theorem}\label{thm1}
 Let $\vec{N}=(N_1,...,N_{r+1})\in\mathbb{Z}^{r+1}$ be such that $N:=\sum\limits_{j=1}^{r+1}N_j\deg\p_j>0$.
Then for all integers $w,m\geqslant 1$ with $wm\geqslant 2$, we have
$$
V_m^w(\vec{N})=\frac{q^{m(N+1-g_K)}}{\zeta_S(wm)}+O_K(q^{N/w}),
$$
where the constant involved in $O_K$ only depends on $K$ and can be determined\linebreak explicitly,
 $g_K$ is the genus of $K$, and $\zeta_S$ is the $S$-zeta function defined by
$$
\zeta_S(s)=\prod_{\mathfrak{p}\in M_K\setminus S}(1-q^{-s\deg\mathfrak{p}})^{-1}\quad (\mathrm{Re}(s)>1).
$$
\end{theorem}

\begin{remark} {\rm (1)} Note that $s=1$ is a pole of $\zeta_S$, hence Theorem \ref{thm1} fails in the case
that $w=m=1$, which corresponds to the counting in $\co_S^{\times}$ (the unit group of $\co_S$).
Put $d_j=\deg\p_j$ and $n_j=v_{\p_j}(a)\ (1\leqslant j\leqslant r+1)$ with $a\in \co_S^{\times}$.
Since $\co_S^{\times}\simeq\mathbb{Z}^r\times \mathbb{F}_q^{\times}$ (see Proposition 14.2 in \cite[p.\,243]{rosen}),
then we have
$$
V_1^1(\vec{N})=(q-1)\mathrm{Card}\Big(\Big\{(n_j)_{1\leqslant j\leqslant r+1}\in \mathbb{Z}^{r+1}:
n_j\geqslant -N_j,\ \sum_{k=1}^{r+1}d_k n_k=0\Big\}\Big).
$$
By Theorem 5.1 in \cite[p.\,70]{lang}, one can show
$$
V_1^1(\vec{N})=\frac{q-1}{\kappa_S}\frac{\gcd(d_1,\ldots, d_{r+1})}{\prod\limits_{j=1}^{r+1}d_j}\frac{N^r}{r!}+O(N^{r-1}).
$$
Here $\kappa_S=\mathrm{Card}\big(\mathrm{Div}^0(S)/\mathscr{P}(S)\big)$ is finite (see
\cite[p.\,243, Proposition 14.1]{rosen}),
$\mathrm{Div}(S)$ is the subgroup of $\mathrm{Div}_K$ generated by $S$, and
\begin{eqnarray*}
\mathrm{Div}^0(S)&=&\{D\in \mathrm{Div}(S): \deg D=0\},\\
\mathscr{P}(S)&=&\{D\in \mathrm{Div}(S): \textrm{$D$ is principal}\}.
\end{eqnarray*}
If we change the manner to count $\co_S^{\times}$, we can obtain a much simple formula.
Note that by taking the generators of $\co_S^{\times}\simeq\mathbb{Z}^r\times \mathbb{F}_q^{\times}$,
each $a\in \co_S^{\times}$ can be associated with a unique element $(m_1(a),\ldots, m_r(a), \eta(a))\in \mathbb{Z}^r\times \mathbb{F}_q^{\times}$. For all integers $k\geqslant 1$, put
$$
V(k)=\mathrm{Card}\big(\big\{a\in \co_S^{\times}: |m_j(a)|\leqslant k\ (1\leqslant j\leqslant r)\big\}\big).
$$
Then we have $V(k)=(2k+1)^r(q-1)$.

{\rm (2)} The case that $m=1$ and $w=2$
can also be treated by a different argument with better error estimate (see \cite{liuyao}).
\end{remark}

In a similar manner as in the real setting, one can extend the notion of $w$-coprime to integral ideals in $\co_S$.

\begin{definition} Let $\mathfrak{a}_1,\dots,\mathfrak{a}_m$ be integral  ideals in $\co_S$, and let $\langle \mathfrak{a}_1,\dots,\mathfrak{a}_m\rangle_S$ be the ideal generated by $\mathfrak{a}_1,\dots,\mathfrak{a}_m$ in $\co_S$.
\begin{itemize}
\item[{\rm (1)}] If $\langle\mathfrak{a}_1,\dots,\mathfrak{a}_m\rangle_S=\co_S$,
then we say that $\mathfrak{a}_1,\dots,\mathfrak{a}_m$ are coprime.

\item[{\rm (2)}] Let $w\geqslant 1$ be an integer. If $\langle \mathfrak{a}_1,\dots,\mathfrak{a}_m\rangle_S$ is not contained in $\q^w$ for any prime ideal $\q$ in $\co_S$, then we say that $\mathfrak{a}_1,\dots,\mathfrak{a}_m$ are $w$-coprime.
\end{itemize}
\end{definition}

For all integers $n\geqslant 1$, we denote by
$Q_{S,m}^w(n)$ the number of $m$-tuples $(\mathfrak{a}_1,\dots,\mathfrak{a}_m)$ of $w$-coprime integral  ideals in $\co_S$ with $\deg\mathfrak{a}_j\leqslant n$ ($1\leqslant j\leqslant m)$.
\vskip 5pt

With the above notation and definitions, we have the following result which is an ideal analogue of Theorem \ref{thm1}.

\begin{theorem}\label{thm2}
Let $w,m\geqslant 1$ be integers such that $wm\geqslant 2$. Then
\begin{eqnarray*}
Q_{S,m}^w(n)=\frac{c_{K,S}^mq^{mn}}{\zeta_S(wm)}+
\begin{cases}
O_{K,S}(q^{n/w}),       & \text{if}\quad m=1,\\
O_{K,S}(nq^{n}),                & \text{if}\quad m=2,w=1,\\
O_{K,S}(q^{n(m-1)}),                        &\text{otherwise}.
\end{cases}
\end{eqnarray*}
Here the constant $c_{K,S}$ is defined by
\begin{eqnarray}
c_{K,S}=\frac{h_Kq^{2-g_K}}{(q-1)^2}\prod\limits_{j=1}^{r+1}\big(1-q^{-\deg\p_j}\big),\label{cks}
\end{eqnarray}
and $h_K$ is the class number of $K$.
\end{theorem}

\begin{remark} For the same reason as in Theorem \ref{thm1}, the above result fails in the case
that $w=m=1$, which corresponds to the integral ideal $\co_S$, and thus $Q_{S,1}^1(n)=1$.
\end{remark}

\subsection{Some preliminary results}

 In this section, we shall present some preparatory results. We begin with the inclusion-exclusion principle in positive characteristic
 which is an exact analog for the classical result (see for example \cite[p.\,39]{tenenbaum}).

Let $\mathrm{Div}_S$ be the subgroup generated by $M_K\setminus S$ in $\mathrm{Div}_K$.
For all divisors
$$
D=\sum\limits_{\p\in M_K\setminus S}c_{\p}\p\in \mathrm{Div}_S,
$$
put $\supp D=\big\{\p:c_{\p}\neq 0\big\}$, and we say that $D$ is squarefree if for all
$\p\in \supp D$, we have $c_{\p}=1$. Finally we define
\begin{eqnarray*}
\mu_S(D)=\left\{\begin{array}{cl}
1, &\textrm{if $D=0$},\\
(-1)^t, &\textrm{if $t=\mathrm{Card}\big(\supp D\big)$ and $D$ is squarefree},\\
0, &\textrm{otherwise},
\end{array}
\right.
\end{eqnarray*}
and call $\mu_S$ the M\"{o}bius function over $\mathrm{Div}_S$. Note that $\mu_S$ is multiplicative, i.e.,
if $D_1,D_2\in\mathrm{Div}_S$ are such that  $\supp D_1\cap\, \supp D_2=\emptyset$, then
$$
\mu_S(D_1+D_2)=\mu_S(D_1)\mu_S(D_2).
$$
Moreover for all $D\in\mathrm{Div}_S$, we define
$$
\delta(D)=\left\{\begin{array}{cl}
1, &\textrm{if}\ D=0,\\
0, &\textrm{otherwise}.
\end{array}\right.
$$
Then for all effective divisors $D$, we have the inverse M\"obius formula
$$
\sum\limits_{0\leqslant D'\leqslant D}\mu_S(D')=\delta(D).
$$
The equality can be proved as in the classical case by the multiplicativity of $\mu_S$
 and the fact $\sum\limits_{j=0}^n\mu_S(j\p)=\delta (n\p)$, for all $\p\in M_K\setminus S$ and for all integers $n\geqslant 0$.

Let $\mathscr{A}$ be a finite set, and let $\mathscr{P}=\{P_j:1\leqslant j\leqslant k\}$ be a set of properties, i.e., each $P_j$ is a subset of $\mathscr{A}$ and the fact that $a\in\mathscr{A}$ satisfies the property $P_j$ means that $a\in P_j$. For each subset $I$ of $\mathscr{P}$, denote by $\mathscr{A}(I)$ the number of all elements $a\in\mathscr{A}$ which satisfy all the properties in $I$, i.e.,
$$
\mathscr{A}(I)=\mathrm{Card}\Big(\bigcap_{P\in I} P\Big).
$$
In particular, we have $A(\emptyset)=\mathrm{Card}(\mathscr{A})$. Let $S(\mathscr{A},\mathscr{P})$ be the number of all elements $a\in \mathscr{A}$ which do not satisfy any property in~$\mathscr{P}$, i.e.,
$$
S(\mathscr{A},\mathscr{P})=\mathrm{Card}\Big(\mathscr{A}\setminus \Big(\bigcup_{j=1}^kP_j\Big)\Big).
$$
 Finally we fix $k$ distinct prime divisors $\q_1,\q_2,\ldots, \q_k$ in $M_K\setminus S$,
 and we associate each property $P_j$ with $\q_j\ (1\leqslant j\leqslant k)$.

\begin{lemma}\label{lem1}
With the above notation, we have
$$
S(\mathscr{A},\mathscr{P})=\sum_{I\subseteq \mathscr{P}}\mu_S(D_I)\mathscr{A}(I),
$$
where for all $I\subseteq \mathscr{P}$, we put $D_I:=\sum\limits_{P_j\in I}\q_j$.
\end{lemma}
\begin{proof} Put $P=\sum\limits_{j=1}^k\q_j$.
For all $a\in \mathscr{A}$, define
$$
F(a)=\sum\limits_{a\in P_j,\atop 1\leqslant j\leqslant k} \q_j.
$$
Then $0\leqslant F(a)\leqslant P$.
Recall that $a\in P_j$ means that $a$ satisfies the property $P_j$. Hence
$F(a)=0$ if and only if $a$ does not satisfy any property  $P_j\ (1\leqslant j\leqslant k)$, i.e., $a$ is counted for $S(\mathscr{A},\mathscr{P})$. Thus
\begin{eqnarray*}
S(\mathscr{A},\mathscr{P})=\sum_{a\in \mathscr{A}}\delta(F(a))=\sum_{a\in \mathscr{A}}\sum_{0\leqslant D\leqslant F(a)}\mu_S(D)
=\sum_{0\leqslant D\leqslant P}\mu_S(D)\sum_{a\in \mathscr{A},\atop D\leqslant F(a)}1.
\end{eqnarray*}
Note that for all effective divisors $D$, we have $D\leqslant F(a)$ if and only if
$$
\mathrm{Supp}D\subseteq \mathrm{Supp}F(a)\subseteq \mathrm{Supp}P,\ \textrm{and $D$ is squarefree},
$$
i.e., $D=D_I$ for some $I\subseteq \mathscr{P}$, from which we deduce directly
\begin{eqnarray*}
S(\mathscr{A},\mathscr{P})&=&\sum_{I\subseteq \mathscr{P}}\mu_S(D_I)\sum_{a\in \mathscr{A},\atop a\in P_j\in I}1=\sum_{I\subseteq \mathscr{P}}\mu_S(D_I)\mathscr{A}(I).
\end{eqnarray*}
Hence the desired result holds.
\end{proof}

\begin{remark}
By Theorem 14.5 in \cite[p.\,249]{rosen}, we can identify the nonzero prime ideals in~$\co_S$ with the prime divisors in $M_K\setminus S$, and for such a prime ideal (resp. prime divisor)~$\p$,
we have the same localization  $\mathscr{O}_{S,\mathfrak{p}}=\mathscr{O}_{\mathfrak{p}}$ (see \cite[pp.\,247--248]{rosen}).
Then a nonzero integral ideal $\m$ in $\co_S$ can be identified with the following divisor
\begin{equation}\label{eq:3}
D(\m)=\sum\limits_{\p\in M_K\setminus S}m_{\p}\p.
\end{equation}
Here $m_{\p}\geqslant 0$ is the integer (noted $v_{\p}(\m)$) determined by $\m\co_{\p}=\pi^{m_{\p}}_{\p}\co_{\p}$,
and $\pi_{\p}$ is a prime element in $\co_{\p}$.
The norm of $\m$ is $\mathrm{N}(\m)=\mathrm{Card}\big(\co_S/\m\big)=q^{\deg\m}$,
and the degree $\deg\m$ of $\m$ satisfies
\begin{eqnarray}\label{eq:4}
\deg\m=\deg D(\m)=\sum\limits_{\p\in M_K\setminus S} m_{\p}\deg \p.
\end{eqnarray}
In particular, we have $\deg\co_S=0$.
By convention, the degree of the zero ideal is $+\infty$.
Finally we define $\mu_S(\m):=\mu_S(D(\m))$. Then we have
\begin{eqnarray*}
\mu_S(\m) = \left\{\begin{array}{cl}
    1,                  &   \text{if } \m=\mathscr{O}_S, \\
    (-1)^t,           &\text{if } \m = \mathfrak{n}_1\cdots\mathfrak{n}_t \text{ with } \mathfrak{n}_1,\dots,\mathfrak{n}_t \text{ distinct prime ideals in $\mathscr{O}_S$}, \\
    0,              & \text{otherwise},
\end{array}\right.
\end{eqnarray*}
and we call $\mu_S$ the M\"{o}bius function over $\co_S$. By the same argument, one can see that the inclusion-exclusion principle also holds for this new $\mu_S$, namely with the notation before Lemma \ref{lem1}, we have
\begin{eqnarray}\label{eq5}
S(\mathscr{A},\mathscr{P})=\sum_{I\subseteq \mathscr{P}}\mu_S(\m_I)\mathscr{A}(I).
\end{eqnarray}
Here $\m_I=\prod\limits_{P_j\in I}\q_j$, for all $I\subseteq \mathscr{P}$, and $\q_1,\ldots, \q_k$ are distinct prime ideals in $\mathscr{O}_S$.
\end{remark}

 For all $D\in \mathrm{Div}_K$, define
$$
L(D)=\big\{a\in K^{\times}: (a)+D\geqslant 0\big\}\cup\{0\}.
$$
Then $L(D)$ is an $\mathbb{F}_q$-vector space with finite dimension, and we denote its dimension by $\ell(D)$.
The famous Riemann-Roch theorem can be reformulated as follows (see for example \cite[Theorem 5.4, p.\,49]{rosen}).

 \begin{lemma}\label{lem2} Let $\mathscr{C}\subset \mathrm{Div}_K$ be the class of canonical divisors. Then there exists an integer $g_K\geqslant 0$ (called the genus of $K$) such that for all $C\in\mathscr{C}$ and for all $D\in \mathrm{Div}_K$, we have
 $$
 \ell(D)=\deg D-g_K+1+\ell(C-D).
 $$
In particular, if $\deg D\geqslant 2g_K-1$, then $\ell(D)=\deg D-g_K+1$.
\end{lemma}

We need also the famous Weil theorem (see for example \cite[pp.\,53-55]{rosen}).

\begin{lemma}\label{lem3} Let $g_K$ be the genus of $K$, and let
$\zeta_K(s)=\prod\limits_{\p\in M_K}\big(1-q^{-s\deg\p}\big)^{-1}$ be the zeta function over $K$. Then we have
$$
\zeta_K(s)=\frac{P_K(q^{-s})}{(1-q^{-s})(1-q^{1-s})}.
$$
Here $P_K(u)$ is a polynomial with integer coefficients such that
$$
P_K(u)=\prod_{j=1}^{2g_K}(1-\pi_ju),
$$
and $\pi_j\in\mathbb{C}$ are such that $|\pi_j|=q^{1/2}\ (1\leqslant j\leqslant 2g_K)$. In particular,
we have
$$
q^{g_K}P_K(q^{-1})=P_K(1)=h_K,
$$
where $h_K$ is
 the class number of $K$.
\end{lemma}

\begin{lemma}\label{lem4}
Let $n\geqslant 0$ be an integer, and let $j_S(n)$ be the number of integral ideals in~$\co_S$ with degree $\leqslant n$.
Then we have
\begin{equation}
j_S(n)=c_{K,S}q^n+O_{K,S}(1),
\end{equation}
where the constant $c_{K,S}$ is given by the formula (\ref{cks}).
\end{lemma}
\begin{proof} Recall that the $S$-zeta function can be also defined as follows:
$$
\zeta_S(s)=\sum_{\mathfrak{a}}\frac{1}{(\mathrm{N}(\mathfrak{a}))^s}\quad (\mathrm{Re}(s)>1),
$$
where $\mathfrak{a}$ varies over all integral ideals of $\co_S$. For all integers $k\geqslant 0$, denote by $b_k$ (resp. $b_{S,k}$)
the number of effective divisors in $\mathrm{Div}_K$ (resp. $\mathrm{Div}_S$) with degree $k$.
Then $b_{S,k}$ is also the number of integral ideals of degree $k$ in $\mathscr{O}_S$. Thus we have
\begin{eqnarray*}
\zeta_K(s)&=&\prod\limits_{\p\in M_K}\big(1-q^{-s\deg\p}\big)^{-1}=\sum_{k=0}^{\infty}b_{k}q^{-sk},\\
\zeta_S(s)&=&\prod_{\mathfrak{p}\in M_K\setminus S}(1-q^{-s\deg\mathfrak{p}})^{-1}=\sum_{k=0}^{\infty}b_{S,k}q^{-sk},
\end{eqnarray*}
from which we obtain by Lemma \ref{lem3}, with $z=q^{-s}$,
$$
\zeta_S(s)=\zeta_K(s)\prod_{j=1}^{r+1}(1-q^{-s\deg\mathfrak{p}_j})=\frac{P_K(z)}{(1-z)(1-qz)}\prod_{j=1}^{r+1}(1-z^{\deg\mathfrak{p}_j}).
$$
Note that $\zeta_K(s)=\sum\limits_{k=0}^{\infty}b_{k}q^{-sk}$ converges absolutely for $\mathrm{Re}(s)>1$,
and that $0\leqslant b_{S,k}\leqslant b_k$ for all integers $k\geqslant 0$.
Thus $\sum\limits_{k=0}^{\infty}b_{S,k}q^{-sk}$ also converges absolutely for $\mathrm{Re}(s)>1$. Then
\begin{eqnarray*}
j_S(n)&=&\sum_{k=0}^{n}b_{S,k}=\frac{1}{2\pi i}\oint_{|z|=q^{-2}}\Big(\sum\limits_{k=0}^{\infty}b_{S,k}z^{k}\Big)\frac{z^{-n-1}\mathrm{d}z}{1-z}\\
&=&\frac{1}{2\pi i}\oint_{|z|=q^{-2}}\frac{P_K(z)}{(1-z)(1-qz)}\Big(\prod_{j=1}^{r+1}(1-z^{\deg \mathfrak{p}_j})\Big)\frac{z^{-n-1}\mathrm{d}z}{1-z}.
\end{eqnarray*}
Indeed the contour $|z|=q^{-2}$ can be replaced by $|z|=q^{-\alpha}$ with any $\alpha>1$,
since $\zeta_S(s)$ converges absolutely for $\mathrm{Re}(s)>1$.

By moving the integral contour to $|z|=q$, we obtain
\begin{eqnarray*}
j_S(n)&=&-\mathrm{Res}_{z=q^{-1}}\frac{P_K(z)}{(1-z)(1-qz)}\Big(\prod_{j=1}^{r+1}(1-z^{\deg \mathfrak{p}_j})\Big)\frac{z^{-n-1}}{1-z}\nonumber\\
       && -\mathrm{Res}_{z=1}\frac{P_K(z)}{(1-z)(1-qz)}\Big(\prod_{j=1}^{r+1}(1-z^{\deg \mathfrak{p}_j})\Big)\frac{z^{-n-1}}{1-z}\nonumber\\
       && +\frac{1}{2\pi i}\oint_{|z|=q}\frac{P_K(z)}{(1-z)(1-qz)}\Big(\prod_{j=1}^{r+1}(1-z^{\deg \mathfrak{p}_j})\Big)\frac{z^{-n-1}\mathrm{d}z}{1-z}\nonumber\\
&=&\frac{h_Kq^{n-g_K+2}}{(q-1)^2}\prod_{j=1}^{r+1}(1-q^{-\deg \mathfrak{p}_j})+R_1+R_2\\
&=& c_{K,S}q^n+R_1+R_2.
\end{eqnarray*}
Here  $P_K(q^{-1})=h_Kq^{-g_K}$ by Lemma \ref{lem3},  and
\begin{eqnarray*}
R_1&=& -\mathrm{Res}_{z=1}\frac{P_K(z)}{(1-z)(1-qz)}\Big(\prod_{j=1}^{r+1}(1-z^{\deg \mathfrak{p}_j})\Big)\frac{z^{-n-1}}{1-z},\\
R_2&=& \frac{1}{2\pi i}\oint_{|z|=q}\frac{P_K(z)}{(1-z)(1-qz)}\Big(\prod_{j=1}^{r+1}(1-z^{\deg \mathfrak{p}_j})\Big)\frac{z^{-n-1}\mathrm{d}z}{1-z}.
\end{eqnarray*}

For the remainder term $R_1$, we distinguish two different cases.
\vskip 5pt

{\bf Case 1:} $\mathrm{Card}(S)=1$, i.e. $S=\{\mathfrak{p}_1\}$. Then
\begin{equation*}
|R_1|\leqslant \frac{P_K(1)\deg\mathfrak{p}_1}{q-1}=\frac{h_K\deg\mathfrak{p}_1}{q-1}.
\end{equation*}

{\bf Case 2:} $\mathrm{Card}(S)>1$. Then $R_1=0$.
\vskip 5pt

For the remainder term $R_2$, we have
\begin{equation*}
|R_2|\leqslant \frac{P_K(q)}{(q-1)(q^2-1)}\Big(\prod_{j=1}^{r+1}(1+q^{\deg\mathfrak{p}_j})\Big)\frac{q^{-n}}{q-1}=o(1)\ (n\to \infty).
\end{equation*}

In conclusion, the desired result holds.
\end{proof}

\begin{remark} By convention, for all integers $n<0$, we put $j_S(n)=0$.
\end{remark}

\subsection{Proof of Theorem \ref{thm1}}
Let $\mathscr{D}$ be the set of $\p$ in $M_K\setminus S$ such that $\deg \p\leqslant N/w$. Then $\mathscr{D}$ is finite.
Set $\mathscr{A}=\big(L(D(\vec N))\big)^m$. For all $\p\in \mathscr{D}$, define
\begin{eqnarray*}
P_{\p}&=&\big\{(a_1,...,a_m)\in \mathscr{A}: D(\langle a_1,...,a_m\rangle_S)\geqslant w\p\}\\
      &=&\big\{(a_1,...,a_m)\in \mathscr{A}: (a_l)\geqslant w\p\ (1\leqslant l\leqslant m)\}.
\end{eqnarray*}
Put $\mathscr{P}=\{P_{\p}: \p\in\mathscr{D}\}$. For all $\p\in\mathscr{D}$ and all $\mathbf{a}=(a_1,...,a_m)\in \mathscr{A}$, we say that $\mathbf{a}$ satisfies the property $P_{\p}$ if $\mathbf{a}\in P_{\p}$. For all $I\subseteq \mathscr{D}$, put $D_I=\sum\limits_{\p\in I}\p$. Then
\begin{eqnarray*}
\mathscr{A}(I)&=&\mathrm{Card}\Big(\Big\{\mathbf{a}=(a_1,...,a_m)\in \mathscr{A}:\ \textrm{$\mathbf{a}$
satisfies the property $P_{\p}$},\ \forall \p\in I\Big\}\Big)\\
    &=&\mathrm{Card}\Big(\Big\{(a_1,...,a_m)\in \mathscr{A}:\ (a_l)\geqslant w\p\ (1\leqslant l\leqslant m),\ \forall \p\in I\Big\}\Big)\\
    &=&\mathrm{Card}\Big(\Big\{(a_1,...,a_m)\in \mathscr{A}:\ (a_l)\geqslant wD_I\ (1\leqslant l\leqslant m)\Big\}\Big)\\
    &=&\Big(\mathrm{Card}\Big(\big\{a\in K^{\times}:\ (a)+D(\vec{N})\geqslant wD_I\big\}\cup \big\{0\big\}\Big)\Big)^m\\
    &=&\mathrm{Card}\Big(\big(L\big(D(\vec N)-wD_I\big)\big)^m\Big)\\
    &=&q^{m\ell(D(\vec{N})-wD_I)}.
\end{eqnarray*}
Note that if $w\deg D_I>N=\deg D(\vec{N})$, then by Lemma 5.3 in \cite[p.\,48]{rosen}, we have
$$
\ell\big(D(\vec{N})-wD_I\big)=0,
$$
and thus $\mathscr{A}(I)=0$. Hence by Lemma \ref{lem1}, we have
\begin{eqnarray}
V_m^w(\vec{N})&=&S(\mathscr{A},\mathscr{P})=\sum_{I\subseteq \mathscr{D}}\mu_S(D_I)\mathscr{A}(I)\\
&=&\sum_{I\subseteq \mathscr{D},\atop \deg D_I\leqslant N/w}\mu_S(D_I)\mathscr{A}(I)\\
&=&\sum_{D\in\mathrm{Div}_S,\, D\geqslant 0,\atop \deg D\leqslant N/w}\mu_S(D)q^{m\ell(D(\vec{N})-wD)}.\label{5}
\end{eqnarray}
Indeed for all $D\in\mathrm{Div}_S$ with $D\geqslant 0$ and $\deg D\leqslant N/w$,
we have $\mu_S(D)\neq 0$ if and only if $D=D_I$ for some $I\subseteq \mathscr{D}$.

Now fix $D\in\mathrm{Div}_S$ satisfying $D\geqslant 0$ and $\deg D\leqslant N/w$.
By Lemma \ref{lem2}, we have
\begin{eqnarray*}
\ell\big(D(\vec{N})-wD\big)=N-w\deg D-g_K+1+\ell\big(C-D(\vec{N})+wD\big).
\end{eqnarray*}
If $\deg \big(D(\vec{N})-wD\big)\geqslant 2g_K-1$, then by Lemma \ref{lem2}, we obtain
\begin{eqnarray}
\ell\big(D(\vec{N})-wD\big)=N-w\deg D-g_K+1.\label{7}
\end{eqnarray}
If $\deg \big(D(\vec{N})-wD\big)\leqslant 2g_K-2$, then $g_K\geqslant 1$, since we have
$$
\deg \big(D(\vec{N})-wD\big)=N-w\deg D\geqslant 0.
$$
Hence there exists a divisor $D_0\geqslant 0$ such that $\deg D_0=2g_K-1$ (see \cite[p.\,52]{rosen}).
Note that $L\big(D(\vec N)-wD\big)\subseteq L\big(D(\vec N)-wD+D_0\big)$, and we have
$$
\deg\big(D(\vec{N})-wD+D_0\big)\geqslant \deg(D_0)=2g_K-1.
$$
Thus by Lemma \ref{lem2}, we have
 \begin{eqnarray}
\ell\big(D(\vec{N})-wD\big)&\leqslant & \ell\big(D(\vec{N})-wD+D_0\big)\\
&=&\deg \big(D(\vec{N})-wD+D_0\big)-g_K+1\\
&=&N-w\deg D+g_K\leqslant  3g_K-2.\label{10}
\end{eqnarray}

For an integer $n\geqslant 0$, denote by $b_{S,n}$
the number of effective divisors in $\mathrm{Div}_S$ of degree $n$.
By the estimate in \cite[p.\,52]{rosen}), for all integers $n>2g_K-2$, we have
\begin{eqnarray}\label{eqn3}
0\leqslant b_{S,n}\leqslant h_K\frac{q^{n-g_K+1}-1}{q-1}=O_K(q^n).
\end{eqnarray}

Recall here that by the formula (\ref{2}),
we have $\mathrm{N}(D)=q^{\deg D}$, for all $D\in\mathrm{Div}_K$.
Recall also that as in the canonical case, we have
\begin{eqnarray}
\sum\limits_{D\in\mathrm{Div}_S,\, D\geqslant 0}\frac{\mu_S(D)}{(\mathrm{N}(D))^{wm}}=\frac{1}{\zeta_S(wm)}.
\end{eqnarray}
Then from the formulas (\ref{5}), (\ref{7}), (\ref{10}), and (\ref{eqn3}),  we deduce
\begin{eqnarray*}
&&V_m^w(\vec{N}) =\sum_{D\in\mathrm{Div}_S,\, D\geqslant 0,\atop \deg D\leqslant N/w}\mu_S(D)q^{m\ell(D(\vec{N})-wD)}\\
&=&\sum_{D\in\mathrm{Div}_S,\, D\geqslant 0,\atop w\deg D\leqslant N-2g_K+1}\mu_S(D)q^{m(N-w\deg D-g_K+1)}
            +\sum_{D\in\mathrm{Div}_S,\, D\geqslant 0,\atop N-2g_K+2\leqslant w\deg D\leqslant N}O_K(1)\nonumber\\
&=&q^{m(N+1-g_K)}\sum_{D\in\mathrm{Div}_S,\, D\geqslant 0,\atop w\deg D\leqslant N-2g_K+1}\frac{\mu_S(D)}{(\mathrm{N}(D))^{wm}}
            +\sum_{D\in\mathrm{Div}_S,\, D\geqslant 0,\atop  w\deg D\leqslant N}O_K(1)\nonumber\\
&=&q^{m(N+1-g_K)}\sum_{D\in\mathrm{Div}_S,\, D\geqslant 0,\atop w\deg D\leqslant N-2g_K+1}\frac{\mu_S(D)}{(\mathrm{N}(D))^{wm}}
            +\sum_{0\leqslant n\leqslant N/w}O_K(b_{S,n})\nonumber\\
&=&q^{m(N+1-g_K)}\sum_{D\in\mathrm{Div}_S,\, D\geqslant 0}\frac{\mu_S(D)}{(\mathrm{N}(D))^{wm}}\\
&&-\frac{q^{mN}}{q^{m(g_K-1)}}\sum_{D\in\mathrm{Div}_S,\, D\geqslant 0, \atop w\deg D>  N-2g_K+1}\frac{\mu_S(D)}{(\mathrm{N}(D))^{wm}}+O_K(q^{N/w})\\
&=&\frac{q^{m(N+1-g_K)}}{\zeta_S(wm)}
-\frac{q^{mN}}{q^{m(g_K-1)}}\sum_{D\in\mathrm{Div}_S,\, D\geqslant 0, \atop  w\deg D>N-2g_K+1}\frac{\mu_S(D)}{(\mathrm{N}(D))^{wm}}+O_K(q^{N/w}).
\end{eqnarray*}

By the same estimate in (\ref{eqn3}) and noting that $wm\geqslant 2$, we obtain
\begin{eqnarray*}
\Big|\sum_{D\in\mathrm{Div}_S,\, D\geqslant 0, \atop w\deg D>N-2g_K+1}\frac{\mu_S(D)}{(\mathrm{N}(D))^{wm}}\Big|&\leqslant &
\sum_{D\in\mathrm{Div}_S,\, D\geqslant 0, \atop w\deg D>N-2g_K+1}q^{-wm\deg D}\\
&=&\sum_{n>(N-2g_K+1)/w}b_{S,n}q^{-wmn}\\
&=&O_K\bigg(\sum_{n>(N-2g_K+1)/w}q^{n-wmn}\bigg)\\
&=&O_K\bigg(\int_{(N-2g_K+1)/w}^{+\infty}q^{(1-wm)x}\mathrm{d} x\bigg)\\
&=& O_K(q^{(1-wm)N/w}),
\end{eqnarray*}
which implies directly the desired result.

\vskip 5pt

\begin{corollary}\label{cor1} Let $w,m\geqslant 1$ be integers such that $wm\geqslant 2$. Then
\begin{eqnarray}
\lim\limits_{N\rightarrow \infty}\frac{V_m^w(\vec{N})}{\kappa_m(\vec{N})}=\frac{1}{\zeta_S(wm)},
\end{eqnarray}
where $\kappa_m(\vec{N})=\mathrm{Card}\big((L(D(\vec{N}))^m\big)=q^{m\ell(D(\vec{N}))}$.
\end{corollary}
\begin{proof} If $\deg D(\vec{N})=N\geqslant 2g_K-1$, then by Lemma \ref{lem2}, we have
\begin{eqnarray*}
\ell(D(\vec{N}))=N+1-g_K.
\end{eqnarray*}
Hence $\kappa_m(\vec{N})=q^{m(N+1-g_K)}$, and the conclusion comes from Theorem \ref{thm1}.
\end{proof}

\begin{remark}\label{remark5}
Corollary \ref{cor1} means that the probability for $m$ elements in $\co_S$ chosen at random to be $w$-coprime is equal to $\frac{1}{\zeta_S(wm)}$.
It is worth noting that the case that $w=1$ has already been treated in \cite{MS} by a quite different method.
\end{remark}

As pointed out in the introduction, a function field Schanuel's theorem was observed by Serre (see for example \cite[p.\,19]{serre})
and proved independently by Dipippo \cite{Dipi}  and Wan (see \cite[Corollary 4.3]{wandaqing}) with a better error term. With the help of geometry of numbers, very recently Phillips
obtained a new proof of the above result but with a weak error term (see \cite[Theorem 4.1]{phi}).
 Below we apply Theorem \ref{thm1} to exhibit another new and elementary proof (see Corollary \ref{cor:2}).
	
	We recall below some definitions and notation.
	
	Fix an integer $n\geqslant 1$.  For all $\mathbf{x}=(x_0:x_1:\cdots:x_n)\in \mathbb{P}^n(K)$,
the (logarithmic) height of $\mathbf{x}$ is defined by
\begin{equation}\label{eq:height-def}
		\mathbf{h}_K(\mathbf{x})=-\sum_{\mathfrak{p}\in M_K}\big(\min_{0\leqslant j\leqslant n}v_{\mathfrak{p}}(x_j)\big)\deg\mathfrak{p}.
\end{equation}
Equivalently, let $\mathrm{inf}_i(x_i)$ denote the greatest divisor $D$
satisfying $D \leqslant (x_j)$ for all integers $j\ (0\leqslant j\leqslant n)$,
where $(x_j)$ is the principal divisor defined by $x_j$,
then
\begin{eqnarray}
\mathbf{h}_K(\mathbf{x})=-\deg \mathrm{inf}_i(x_i).
\end{eqnarray}
By Proposition 5.1 in \cite[p.\,47]{rosen}, for all $y\in K^{\times}$,
we have
\begin{eqnarray}\label{19}
\sum\limits_{\mathfrak{p}\in M_K}v_{\mathfrak{p}}(y)\deg\mathfrak{p}=\deg (y)=0,
\end{eqnarray}
thus $\mathbf{h}_K(\mathbf{x})$ is independent of the choice of $(x_j)_{0\leqslant j\leqslant n}$.
Finally for all integers $B>0$, let $N_K(B,n)$ denote the number of all $\mathbf{x}\in \mathbb{P}^n(K)$
	satisfying $\mathbf{h}_K(\mathbf{x})\leqslant B$.
	
With the notation as above, we have the following result.

\begin{corollary}\label{cor:2}
Let $n\geqslant 1$ be an integer. Then
\begin{eqnarray}
N_K(B,n)=\frac{h_Kq^{(n+1)(B-g_K+1)}}
		{(q-1)\zeta_K(n+1)}\cdot\frac{1}{1-q^{-n-1}}+O_K(q^{B}).
\end{eqnarray}
\end{corollary}
\begin{proof}
	Fix $\infty\in M_K$. Put $d=\deg\infty$ and $S=\{\infty\}$. Then $r=\mathrm{Card}(S)-1=0$,
and $\co_S^{\times}\simeq\mathbb{F}_q^{\times}$ (see \cite[p.\,243, Proposition 14.2]{rosen}).
Let $\mathrm{Cl}_S$ be the ideal class group of $\co_S$, and
put $h_S=\mathrm{Card}(\mathrm{Cl}_S)$, the class number of $\co_S$.
Let $\mathfrak{c}_1, \ \mathfrak{c}_2, \ldots, \mathfrak{c}_{h_S}$ be integral ideals of
$\co_S$, which form a complete system of representatives in $\mathrm{Cl}_S$.
By the Chebotarev density theorem (see for example \cite[p.\,125, Theorem 9.13A]{rosen}),
each ideal class contains infinitely prime ideals. Thus we can also suppose that
all the $\mathfrak{c}_j$ are prime and different from $\co_S$.
Let $\p_j$ be the prime divisor associated with $\mathfrak{c}_j$, i.e., $\p_j=D(\mathfrak{c}_j)\ (1\leqslant j\leqslant h_S)$.
Then all the $\p_j$ are different from $\infty$.	
	
For all integers $m\ (0\leqslant m\leqslant n)$,
let $\mathbb{P}^{n,m}(K)$
be the set of all points in $\mathbb{P}^{n}(K)$ with exactly $m$ zero coordinates.
For all integers $B>0$, let $N_{K,m}(B,n)$ be the number of all $\mathbf{x}\in \mathbb{P}^{n,m}(K)$
with $\mathbf{h}_K(\mathbf{x})\leqslant B$.  Then
\begin{equation}\label{eq:Nsum}
N_K(B,n)=\sum_{m=0}^{n}N_{K,m}(B,n).
\end{equation}

	For all $\mathbf{x}\in \mathbb{P}^{n}(K)$, we can write $\mathbf{x}=(x_0:x_1:\cdots:x_n)$ with $x_j\in\co_S\ (0\leqslant j\leqslant n)$.
Then we can find $y\in K^{\times}$	and a unique integer $k\ (1\leqslant k\leqslant h_S)$ such that
$$
y\langle x_0,x_1,\ldots,x_n\rangle_S=\mathfrak{c}_k.
$$
In particular, we have $yx_j\in \co_S\ (0\leqslant j\leqslant n)$.
Up to replacing $x_j$ by $yx_j$, we can thus suppose $\langle x_0,x_1,\ldots,x_n\rangle_S=\mathfrak{c}_k$. Then
\begin{eqnarray}\label{eq:h}
\mathbf{h}_K(\mathbf{x})&=&-\sum_{\mathfrak{p}\in M_K}\big(\min_{0\leqslant j\leqslant n}v_{\mathfrak{p}}(x_j)\big)\deg\mathfrak{p}\\
&=&-\sum_{\mathfrak{p}\in M_K\setminus S}\big(\min_{0\leqslant j\leqslant n}v_{\mathfrak{p}}(x_j)\big)\deg\mathfrak{p}
-d\big(\min_{0\leqslant j\leqslant n}v_{\infty}(x_j)\big)\\
&=&-\deg \mathfrak{p}_k-d\big(\min_{0\leqslant j\leqslant n}v_{\infty}(x_j)\big).\label{24}
\end{eqnarray}
Here by the formulas (\ref{eq:3}) and (\ref{eq:4}), and by Lemma 4-2-1\,(4) in \cite[p.\,127]{weiss}, we have
\begin{eqnarray}\label{v_p}
v_{\p}(\p_k)=v_{\p}(D(\mathfrak{c}_k))=v_{\p}(\mathfrak{c}_k)=\min\limits_{0\leqslant j\leqslant n}v_{\mathfrak{p}}(x_j),
\end{eqnarray}
for all $\p\in M_K\setminus S$.
Note that by the formula (\ref{19}), for all $y\in K^{\times}$, we have
\begin{eqnarray}\label{25}
-dv_{\infty}(y)=\sum_{\p\in M_K\setminus S}v_{\p}(y)\deg \p=\deg \langle y\rangle_S.
\end{eqnarray}
Then by combining the formulas (\ref{24}) and (\ref{25}) together, we obtain
\begin{eqnarray}\label{26}
\mathbf{h}_K(\mathbf{x})=-\deg \mathfrak{p}_k+\max_{0\leqslant j\leqslant n} \deg \langle x_j\rangle_S.
\end{eqnarray}
Thus $\mathbf{h}_K(\mathbf{x})\leqslant B$ if and only if $\deg \langle x_j\rangle_S\leqslant  B+\deg\p_k\ (0\leqslant j\leqslant n)$.

For all integers $k,\ell$ with $1\leqslant k\leqslant h_S$ and $0\leqslant \ell\leqslant n$,
let $N_{k,\ell}(B)$ be the number of $(a_0,\ldots,a_\ell)\in (\co_S\setminus\{0\})^{\ell+1}$
satisfying $\deg \langle a_j\rangle_S\leqslant B+\deg\p_k\ (0\leqslant j\leqslant \ell)$, and
$$
\langle a_0,\ldots, a_{\ell}\rangle_S=\mathfrak{c}_k.
$$
Note that if $(a_0,\ldots,a_\ell)$ and $(a_0',\ldots,a_\ell')$ are both counted for $N_{k,\ell}(B)$
and there exists some $y\in K^{\times}$ such that $a_j'=ya_j\ (0\leqslant j\leqslant \ell)$,
then by the formula (\ref{v_p}), we have $v_{\p}(y)=0$ for all $\p\in M_K\setminus S$,
thus $y\in \co_S^{\times}\simeq \mathbb{F}_q^{\times}$.
Note also that each point in $\mathbb{P}^{n,m}(K)$ has exactly $n+1-m$ nonzero coordinates,
and choosing $m$ zero coordinates gives $\binom{n+1}{m}$ possibilities, hence
\begin{equation}\label{eq:NKm}
N_{K,m}(B,n)=\frac{1}{q-1}\binom{n+1}{m}\sum_{k=1}^{h_S} N_{k,\,n-m}(B).
\end{equation}

	Put $S_k=\{\mathfrak{p}_k,\infty\}$, and work in the ring of $\co_{S_k}$.

For all integers $m\ (0\leqslant m\leqslant n)$,
and for all $\vec{N}=(N_1,N_2)\in \mathbb{Z}^2$ such that $N:=N_1\deg \p_k+ N_2\deg\infty=N_1\deg \p_k+dN_2>0$, define
$$
\mathbb{W}_{k,m}(\vec{N})=\{(a_0,\ldots,a_m)
\in\big(L(D(\vec N))\setminus\{0\}\big)^{m+1}:
	a_0,\ldots,a_m \text{ are coprime}\big\}.
$$
Then $(a_0,\ldots,a_m)\in \mathbb{W}_{k,m}(\vec{N})$ if and only if  $(a_0,\ldots,a_m)\in \big(\co_{S_k}\setminus\{0\}\big)^{m+1}$ is such that
$\langle a_0,\ldots,a_m\rangle_{S_k}=\co_{S_k}$ and $v_{\p_k}(a_j)\geqslant -N_1,\ v_{\infty}(a_j)\geqslant -N_2\ (0\leqslant j\leqslant m)$.

Fix an integer $B>d+\max\limits_{1\leqslant j\leqslant h_S}\deg \p_j$. For all integers $k\ (1\leqslant k\leqslant h_S)$, define
$$
q_k:=\lfloor\frac{B+\deg \p_k}{d}\rfloor.
$$
Then by the formula (\ref{24}), we obtain that
\begin{eqnarray}\label{28}
\textrm{$\mathbf{h}_K(\mathbf{x})\leqslant B$ if and only if
$\min_{0\leqslant j\leqslant n} v_\infty(x_j)\geqslant -q_k$}.
\end{eqnarray}

Put $\vec{N}(k)=(-1,\,q_k)$, and $\vec{M}(k)=(-2,\,q_k)$. Since $B>d+\max\limits_{1\leqslant j\leqslant h_S}\deg \p_j$, then
$$
N(k):=-\deg \p_k+dq_k>0,\ \textrm{and}\ M(k):=-2\deg\p_k+dq_k>0.
$$
Moreover $\mathbb{W}_{k,m}(\vec{M}(k))\subseteq \mathbb{W}_{k,m}(\vec{N}(k))$, and $(a_0,\ldots, a_m)\in \mathbb{W}_{k,m}(\vec{N}(k))\setminus \mathbb{W}_{k,m}(\vec{M}(k))$ if and only if $(a_0,\ldots,a_m)\in \big(\co_{S_k}\setminus\{0\}\big)^{m+1}$ is such that
$\langle a_0,\ldots,a_m\rangle_{S_k}=\co_{S_k}$, and
$$
v_{\p_k}(a_j)=1,\ v_{\infty}(a_j)\geqslant -q_k\ (0\leqslant j\leqslant m).
$$
Equivalently, this means that $(a_0,\ldots,a_m)\in \big(\co_{S}\setminus\{0\}\big)^{m+1}$ is such that
$$
\langle a_0,\ldots,a_m\rangle_{S}=\mathfrak{c}_k,\ \textrm{and}\ v_{\infty}(a_j)\geqslant -q_k\ (0\leqslant j\leqslant m).	
$$
Hence by  the formulas (\ref{24}) and (\ref{26}),
and the property (\ref{28}), we have
\begin{equation}\label{eq:NtoW}
N_{k,m}(B)=\mathrm{Card}(\mathbb{W}_{k,m}(\vec{N}(k)))-\mathrm{Card}(\mathbb{W}_{k,m}(\vec{M}(k))).
\end{equation}
	
Put $W_{k,m}(\vec{N}(k))=\mathrm{Card}(\mathbb{W}_{k,m}(\vec{N}(k)))$ and $W_{k,m}(\vec{M}(k))=\mathrm{Card}(\mathbb{W}_{k,m}(\vec{M}(k)))$.
By combining all the formulas (\ref{eq:Nsum}), (\ref{eq:NKm}), and (\ref{eq:NtoW}) together, we obtain
\begin{eqnarray}
N_K(B,n)&=&\frac{1}{q-1}\sum_{m=0}^{n}\binom{n+1}{m}\sum_{k=1}^{h_S} \big(W_{k,n-m}(\vec{N}(k))-W_{k,n-m}(\vec{M}(k))\big)\\
&=&\frac{1}{q-1}\sum_{k=1}^{h_S}\sum_{m=0}^{n}\binom{n+1}{m} \big(W_{k,n-m}(\vec{N}(k))-W_{k,n-m}(\vec{M}(k))\big)\\
&=&\frac{1}{q-1}\sum_{k=1}^{h_S}\big(V_{k,n+1}^1(\vec{N}(k))-V_{k,n+1}^1(\vec{M}(k))\big),\label{eq32}
\end{eqnarray}
where $V_{k,n+1}^1(\vec{N}(k))$ and $V_{k,n+1}^1(\vec{M}(k)$ are defined by
\begin{eqnarray*}
V_{k,n+1}^1(\vec{N})=\mathrm{Card}\left(\left\{(a_0,...,a_{n})\in \co_{S_k}^{n+1}: \begin{array}{c}
              \textrm{$a_0,...,a_n$ are coprime, $v_{\mathfrak{p}_k}(a_j)\geqslant 1$,} \\
              v_{\infty}(a_j)\geqslant -q_k\ (0\leqslant j\leqslant n)
              \end{array}\right\}\right),\\
V_{k,n+1}^1(\vec{M})=\mathrm{Card}\left(\left\{(a_0,...,a_{n})\in \co_{S_k}^{n+1}: \begin{array}{c}
              \textrm{$a_0,...,a_n$ are coprime, $v_{\mathfrak{p}_k}(a_j)\geqslant 2$,} \\
              v_{\infty}(a_j)\geqslant -q_k\ (0\leqslant j\leqslant n)
              \end{array}\right\}\right).
\end{eqnarray*}

By applying Theorem \ref{thm1} with $V_{k,n+1}^1(\vec{N}(k))$ and $V_{k,n+1}^1(\vec{M}(k))$, we have
\begin{eqnarray}
V_{k,n+1}^1(\vec{N}(k))&=&\frac{q^{(n+1)(N(k)+1-g_K)}}{\zeta_{S_k}(n+1)}+O_K(q^{N(k)}),\label{33}\\
V_{k,n+1}^1(\vec{M}(k))&=&\frac{q^{(n+1)(M(k)+1-g_K)}}{\zeta_{S_k}(n+1)}+O_K(q^{M(k)}).\label{34}
\end{eqnarray}
Write $B+\deg\mathfrak{c}_k=B+\deg\p_k=d q_k+r_k$, with $0\leqslant r_k<d$. Then
$$
N(k)=B-r_k,\ \textrm{and}\ M(k)=B-\deg\p_k-r_k,
$$
and by the formulas (\ref{33}) and (\ref{34}), we obtain
\begin{eqnarray*}
&&V_{k,n+1}^1(\vec{N}(k))-V_{k,n+1}^1(\vec{M}(k))\\
&=&\frac{q^{(n+1)(B-r_k+1-g_K)}(1-q^{-(n+1)\deg\p_k})}
		{\zeta_{S_k}(n+1)}
		+O_K(q^{B-r_k})\nonumber\\
&=&\frac{q^{(n+1)(B+1-g_K-r_k)}(1-q^{-(n+1)\deg\p_k})}
		{\zeta_K(n+1)(1-q^{-(n+1)\deg\p_k})(1-q^{-(n+1)d})}
		+O_K(q^{B-r_k})\nonumber\\
		&=&\frac{q^{(n+1)(B+1-g_K-r_k)}}
		{\zeta_K(n+1)(1-q^{-(n+1)d})}+O_K(q^{B}).
\end{eqnarray*}
Put	$\beta:=\sum\limits_{k=1}^{h_S}q^{-(n+1)r_k}$. By the above equality and the formula (\ref{eq32}), we have
\begin{eqnarray}
N_K(B,n)&=&\frac{\beta q^{(n+1)(B+1-g_K)}}
		{(q-1)\zeta_K(n+1)(1-q^{-(n+1)d})}+O_K(q^{B}).
\end{eqnarray}
Note that for all integers $k\ (1\leqslant k\leqslant h_S)$, we have $0\leqslant r_k<d$
and
$$
r_k\equiv B+\deg\mathfrak{c}_k\pmod d.
$$
Note also that $\mathfrak{c}_1,\ldots, \mathfrak{c}_{h_S}$ form a complete system of representatives in $\mathrm{Cl}_S$,
 and the degree map $\deg\colon\mathrm{Cl}_{S}\to\mathbb{Z}/d\mathbb{Z}$ is a surjective homomorphism whose kernel contains $h_K$ elements (see \cite[p.\,243, Proposition 14.1]{rosen}).
Hence
$$
\{r_k: 1\leqslant k\leqslant h_S\}=\{0,1,\ldots, d-1\}\
\textrm{and}\ \mathrm{Card}\{k: r_k=j\}=h_K\ (0\leqslant j\leqslant d-1).
$$
Consequently we obtain
$$
\beta=\sum_{j=0}^{d-1} h_K q^{-(n+1)j}
	=h_K\frac{1-q^{-(n+1)d}}{1-q^{-(n+1)}}.
$$
This ends the proof of Corollary \ref{cor:2}.
\end{proof}

\begin{remark} Serre counted $A_K(B,n)$ the number of all $\mathbf{x}\in \mathbb{P}^n(K)$
	with $\mathbf{h}_K(\mathbf{x})=B$. By Corollary \ref{cor:2}, we have
$$
A_K(B,n)=\mathrm{N}_K(B,n)-\mathrm{N}_K(B-1,n)=\frac{h_Kq^{(n+1)(B-g_K+1)}}
		{(q-1)\zeta_K(n+1)}+O_K(q^{B}).
$$
The error term is worse than that of  Dipippo and Wan (which is $O_K(q^{\frac{B}{2}+\varepsilon})$, for all $\varepsilon>0$),
but better than that of Serre and Phillips.
\end{remark}

\subsection{Proof of Theorem 2.}

Fix an integer $n\geqslant 1$. Let $\mathscr{A}$ be the set of all $m$-tuples $(\mathfrak{a}_1,\dots,\mathfrak{a}_m)$ of integral ideals in $\co_S$ with $\deg\mathfrak{a}_j\leqslant n$ ($1\leqslant j\leqslant m)$.
Let $\mathscr{D}$ be the set of all nonzero prime ideals $\q$ in $\co_S$ satisfying $\deg \q\leqslant n/w$.
 For all $\q\in \mathscr{D}$, define
$$
P_{\q}=\big\{(\mathfrak{a}_1,\dots,\mathfrak{a}_m)\in \mathscr{A}: \langle\mathfrak{a}_1,\dots,\mathfrak{a}_m\rangle_S\subseteq  \q^w\big\}.
$$
Put $\mathscr{P}=\{P_{\q}: \q\in\mathscr{D}\}$. For all $\q\in\mathscr{D}$ and all $\mathfrak{a}=(\mathfrak{a}_1,...,\mathfrak{a}_m)\in \mathscr{A}$, we say that $\mathfrak{a}$ satisfies the property $P_{\q}$ if $\mathfrak{a}\in P_{\q}$.
For all $I\subseteq \mathscr{D}$, put $\mathfrak{m}_I=\prod\limits_{\q\in I}\q$. Then
\begin{eqnarray*}
\mathscr{A}(I)&=&\mathrm{Card}\Big(\Big\{(\mathfrak{a}_1,\dots,\mathfrak{a}_m)\in \mathscr{A}:\ \mathfrak{a}_l\subseteq  \q^w\ (1\leqslant l\leqslant m),\ \forall \q\in I\Big\}\Big)\\
    &=&\mathrm{Card}\Big(\Big\{(\mathfrak{a}_1,\dots,\mathfrak{a}_m)\in \mathscr{A}:\ \mathfrak{a}_l\subseteq  \mathfrak{m}_I^w\Big\}\Big)\\
    &=&\Big(\mathrm{Card}\Big(\Big\{\mathfrak{a}:\ \mathfrak{a}\subseteq  \mathfrak{m}_I^w,\ \textrm{$\mathfrak{a}$ is an integral ideal in $\mathscr{O}_S$},\ \deg\mathfrak{a}\leqslant n\Big\}\Big)\Big)^m\\
    &\stackrel{\mathfrak{b}=\mathfrak{a}\mathfrak{m}_I^{-w}}{=\hskip -2pt=\hskip -2pt=\hskip -2pt=\hskip -2pt=\hskip -2pt=}&\Big(\mathrm{Card}\Big(\Big\{\mathfrak{b}:\ \textrm{$\mathfrak{b}$ is an integral ideal in $\mathscr{O}_S$},\ \deg\mathfrak{b}\leqslant n-w\deg\mathfrak{m}_I\Big\}\Big)\Big)^m\\
    &=&\big(j_S(n-w\deg\mathfrak{m}_I)\big)^m.
\end{eqnarray*}
Recall here that as in the canonical case, we also have
$$
\sum_{\mathfrak{b}}\frac{\mu_S(\mathfrak{b})}{q^{mw\deg\mathfrak{b}}}=\frac{1}{\zeta_S(wm)}.
$$
Hence by the formula (\ref{eq5}) and Lemma \ref{lem4}, we have
\begin{eqnarray*}
Q_{S,m}^w(n)&=&S(\mathscr{A},\mathscr{P})=\sum_{I\subseteq \mathscr{D}}\mu_S(\mathfrak{m}_I)\mathscr{A}(I)\\
&=&\sum_{I\subseteq \mathscr{D}}\mu_S(\mathfrak{m}_I)\big(j_S(n-w\deg\mathfrak{m}_I)\big)^m\\
&=&\sum_{\deg\mathfrak{b}\leqslant n/w}\mu_S(\mathfrak{b})\big(j_S(n-w\deg\mathfrak{b})\big)^m\\
&=&\sum_{\deg\mathfrak{b}\leqslant n/w}\mu_S(\mathfrak{b})\big(c_{K,S}q^{n - w\deg\mathfrak{b}}+O_{K,S}(1)\big)^m\nonumber\\
&=&c_{K,S}^mq^{mn}\sum_{\deg\mathfrak{b}\leqslant n/w}\mu_S(\mathfrak{b})q^{-mw\deg\mathfrak{b}}\\
&&+O_{K,S}\bigg(\sum_{\deg\mathfrak{b}\leqslant n/w}q^{(n   -  w\deg\mathfrak{b})(m-1)}\bigg)\nonumber\\
&=&c_{K,S}^mq^{mn}\sum_{\mathfrak{b}}\frac{\mu_S(\mathfrak{b})}{q^{mw\deg\mathfrak{b}}}
          -c_{K,S}^mq^{mn}\sum_{\deg\mathfrak{b}> n/w}\frac{\mu_S(\mathfrak{b})}{q^{mw\deg\mathfrak{b}}}\\
&&          +O_{K,S}\bigg(\sum_{\deg\mathfrak{b}\leqslant n/w}q^{(n-w\deg\mathfrak{b})(m-1)}\bigg)\nonumber\\
&=&\frac{c_{K,S}^mq^{mn}}{\zeta_S(mw)}+R_1+O_{K,S}(R_2).
\end{eqnarray*}
Here the remainder terms $R_1$ and $R_2$ are defined as
\begin{eqnarray*}
R_1&=&-c_{K,S}^mq^{mn}\sum_{\deg\mathfrak{b}> n/w}\frac{\mu_S(\mathfrak{b})}{q^{mw\deg\mathfrak{b}}},\\
R_2&=&\sum_{\deg\mathfrak{b}\leqslant n/w}q^{(n-w\deg\mathfrak{b})(m-1)}.
\end{eqnarray*}

Recall that for all integers $k\geqslant 0$, we denote by $b_{S,k}$
the number of integral ideals of degree $k$ in $\mathscr{O}_S$, which
is also the number of effective divisors of degree~$k$ in $\mathrm{Div}_S$.
Recall also that by the inequality (\ref{eqn3}), we have $b_{S,k}=O_{K}(q^{k})$.
Consequently for the remainder term $R_1$, we have
\begin{eqnarray*}
|R_1|  &=&\Big|c_{K,S}^mq^{mn}\sum_{\deg\mathfrak{b}> n/w}\frac{\mu_S(\mathfrak{b})}{q^{mw\deg\mathfrak{b}}}\Big|\\
&\leqslant &c_{K,S}^mq^{mn}\sum_{\deg\mathfrak{b}> n/w}\frac{1}{q^{mw\deg\mathfrak{b}}}\\
&=&c_{K,S}^mq^{mn}\sum_{k> n/w}\frac{b_{S,k}}{q^{mwk}}\\
&\ll_{K,S} & q^{mn}\sum_{k>n/w}\frac{1}{q^{(mw-1)k}}\\
&=&O_{K,S}(q^{n/w}).
\end{eqnarray*}

For the remainder term $R_2$, we have
\begin{eqnarray*}
R_2 &=&\sum_{\deg\mathfrak{b}\leqslant n/w}q^{(n-w\deg\mathfrak{b})(m-1)}\\
&=&\sum_{k\leqslant n/w}b_{S,k}q^{(n-wk)(m-1)}\\
&\ll_{K,S}&  q^{n(m-1)}\sum_{k \leqslant n/w}q^{-wk(m-1)+k}\\
&\ll_{K,S}&
\begin{cases}
q^{n/w},       & \text{if}\quad m=1,\\
nq^{n},               & \text{if}\quad m=2,w=1,\\
q^{n(m-1)},                        &\text{otherwise}.
\end{cases}
\end{eqnarray*}
Hence the desired result holds.

\begin{corollary}\label{cor2}
Let $w,m\geqslant 1$ be integers such that $wm\geqslant 2$. Then
\begin{eqnarray*}
\lim_{n\rightarrow \infty}\frac{Q_{S,m}^w(n)}{(j_S(n))^m}=\frac{1}{\zeta_S(w m)}.
\end{eqnarray*}
\end{corollary}
\begin{proof} The conclusion comes directly from Theorem \ref{thm2} and Lemma \ref{lem4}.
\end{proof}

\begin{remark} As for Corollary \ref{cor1}, the above result means that the probability for $m$ integral ideals in $\co_S$ chosen at random to be $w$-coprime is equal to $\frac{1}{\zeta_S(w m)}$.
In particular, for $m=1$ and $w=2$, we obtain the exact analog of the classical result that the probability for an integral ideal in $\co_S$ chosen at random to be squarefree is equal to $\frac{1}{\zeta_S(2)}$ (see \cite[Theorem 5.1]{sittinger}).
\end{remark}

\bigskip

\noindent \textbf{Acknowledgments.}
Si-Han Liu would like to warmly thank the Talent Fund of Beijing Jiaotong University
(Grant No.2026JBRC014) and the China Postdoctoral Science Foundation
(Grant Nos.\,GZC20232908 and 2024M753428) for partial financial support.
Zhe-Cheng Liu and Jia-Yan Yao would like to heartily thank
the National Natural Science Foundation of China (Grant No.\,12231013) for partial financial support.
Finally we would like to sincerely thank the anonymous referee for pertinent comments and valuable suggestions.

\smallskip

\bibliographystyle{plain}
\bibliography{biblio}

\vskip 20pt

\begin{tabular}{l}
Si-Han LIU   \\
Department of Mathematics \\
Tsinghua University \\
Beijing 100084 \\
School of Mathematics and Statistics\\
Beijing Jiaotong University\\
Beijing 100044\\
People's Republic of China \\
E-mail: shliu3@bjtu.edu.cn
\end{tabular}%
\vskip 10pt

\begin{tabular}{l}
Zhe-Cheng LIU (Corresponding author)\\
Department of Mathematics \\
Tsinghua University \\
Beijing 100084 \\
People's Republic of China \\
E-mail: liuzc21@mails.tsinghua.edu.cn
\end{tabular}%
\vskip 10pt

\begin{tabular}{l}
Jia-Yan YAO \\
Department of Mathematics\\
Tsinghua University\\
Beijing 100084\\
People's Republic of China\\
E-mail: jyyao@mail.tsinghua.edu.cn
\end{tabular}%

\end{document}